\documentclass[11pt]{article}
\usepackage[utf8]{inputenc}
\usepackage{authblk}
\usepackage{mathtools}
\usepackage[english]{babel}
\usepackage{amsmath}
\usepackage{amsthm}
\usepackage{amsfonts}
\usepackage{amssymb}
\usepackage{graphicx}
\usepackage[usenames,dvipsnames]{color}
\usepackage{hyperref}
\usepackage[left=2cm,right=2cm,top=2cm,bottom=2cm]{geometry}
\usepackage{dsfont}
\usepackage{tikz}
\usepackage{enumerate}
\usepackage[numeric,initials,nobysame,msc-links,abbrev]{amsrefs}
\renewcommand{\eprint}[1]{\href{https://arxiv.org/abs/#1}{arXiv:#1}}
\newcommand{\pageafter}[1]{#1~pp.}
\BibSpec{article}{%
+{} {\PrintAuthors} {author}
+{,} { \textit} {title}
+{.} { } {part}
+{:} { \textit} {subtitle}
+{,} { \PrintContributions} {contribution}
+{.} { \PrintPartials} {partial}
+{,} { } {journal}
+{} { \textbf} {volume}
+{} { \PrintDatePV} {date}
+{,} { \issuetext} {number}
+{,} { \pageafter} {pages}
+{,} { } {status}
+{,} { \PrintDOI} {doi}
+{,} { available at \eprint} {eprint}
+{} { \parenthesize} {language}
+{} { \PrintTranslation} {translation}
+{;} { \PrintReprint} {reprint}
+{.} { } {note}
+{.} {} {transition}
+{} {\SentenceSpace \PrintReviews} {review}
}
\BibSpec{collection.article}{%
+{} {\PrintAuthors} {author}
+{,} { \textit} {title}
+{.} { } {part}
+{:} { \textit} {subtitle}
+{,} { \PrintContributions} {contribution}
+{,} { \PrintConference} {conference}
+{} {\PrintBook} {book}
+{,} { } {booktitle}
+{,} { \PrintDateB} {date}
+{,} { \pageafter} {pages}
+{,} { } {status}
+{,} { \PrintDOI} {doi}
+{,} { available at \eprint} {eprint}
+{} { \parenthesize} {language}
+{} { \PrintTranslation} {translation}
+{;} { \PrintReprint} {reprint}
+{.} { } {note}
+{.} {} {transition}
+{} {\SentenceSpace \PrintReviews} {review}
}

\usetikzlibrary{arrows}
\usepackage[normalem]{ulem}

\usepackage[capitalise]{cleveref}

\newtheorem{theorem}{Theorem}
\crefname{theorem}{Theorem}{Theorems}

\crefname{corollary}{Corollary}{Corollaries}
\newtheorem{lemma}[theorem]{Lemma}
\crefname{lemma}{Lemma}{Lemmas}
\newtheorem{proposition}[theorem]{Proposition}
\crefname{proposition}{Proposition}{Propositions}

\crefname{conjecture}{Conjecture}{Conjectures}
\newtheorem{question}[theorem]{Question}
\crefname{question}{Question}{Questions}
\theoremstyle{definition}
\newtheorem{definition}[theorem]{Definition}
\crefname{definition}{Definition}{Definitions}

\crefname{remark}{Remark}{Remarks}

\crefname{example}{Example}{Examples}

\crefname{observation}{Observation}{Observations}

\crefname{claim}{Claim}{Claims}

\crefname{assumption}{Assumption}{Assumptions}

\newcommand{\cE}{\mathcal{E}}
\newcommand{\cF}{\mathcal{F}}

\title{\scshape The maximal running time of hypergraph bootstrap percolation}

\author[1]{Ivailo Hartarsky}
\author[2]{Lyuben Lichev}
\affil[1]{TU Wien, Faculty of Mathematics and Geoinformation, Institute of Statistics and Mathematical Methods in Economics, Research Unit of Mathematical Stochastics, Wiedner Hauptstra\ss e 8-10, A-1040 Vienna, Austria, \texttt{ivailo.hartarsky@tuwien.ac.at}}
\affil[2]{Jean Monnet University and Institut Camille Jordan, Saint-Etienne and Lyon, France, \texttt{lyuben.lichev@univ-st-etienne.fr}}

\begin{document}

\maketitle

\begin{abstract}
We show that for every $r\ge 3$, the maximal running time of the $K^{r}_{r+1}$-bootstrap percolation in the complete $r$-uniform hypergraph on $n$ vertices $K_n^r$ is $\Theta(n^r)$. This answers a recent question of Noel and Ranganathan in the affirmative, and disproves a conjecture of theirs. Moreover, we show that the prefactor is of the form $r^{-r} \mathrm{e}^{O(r)}$ as $r\to\infty$.
\end{abstract}


\section{Introduction}
\subsection{Background}
Bootstrap percolation is a widely studied model for spreading infection or information on a graph. Given an initial set of infections, this process iteratively adds infections if a certain local pattern of infections is already present. While the original statistical physics perspective suggests selecting initial infections at random, interesting computational and extremal combinatorial questions arise when considering deterministic instances. The most classical extremal problems are determining the minimal number of infections necessary for infecting the entire graph and the maximal time until stationarity or complete infection.

These maximal times have been studied in several settings, including grids \cite{Benevides15} and hypercubes \cites{Przykucki12,Hartarsky18}. Bootstrap percolation in which one infects the edges of a graph when they complete a copy of a fixed graph is referred to graph bootstrap percolation or weak saturation. In that setting, the maximal times were studied in \cites{BaloghNan,Bollobas17a,Matzke15}. Finally, our setting of interest is the more general hypergraph bootstrap percolation, where maximal times were recently investigated by Noel and Ranganathan \cite{Noel23}.

\subsection{Model and result}
For a positive integer $r\ge 2$ and an $r$-uniform hypergraph $H$ (or $r$-graph for short), we identify $H$ with its edge set and denote by $V(H)$ the set of vertices of $H$. In this paper, we pay particular attention to the complete $r$-graph on $n$ vertices denoted $K^r_n$, especially the $r$-graph $K^r_{r+1}$, which we denote $F_r$ for short.

For a positive integer $r\ge 2$ and an $r$-graph $F$, the \emph{$F$-bootstrap process in $K^r_n$} is a monotone cellular automaton described as follows. At the beginning, a subgraph $G_0$ of $K^r_n$ is infected. At every step $i\ge 1$, we define
\begin{equation}\label{eq:G_i}
    G_i = G_{i-1}\cup \{e\in K^r_n\colon \text{there is a copy }F'\subset K^r_n\text{ of }F\text{ such that }e = F'\setminus G_{i-1}\}.
\end{equation}
The \emph{running time} until stationarity (that is, the first step $i$ at which $G_i = G_{i+1}$) of the $F$-bootstrap percolation with initially infected set $G_0$ is denoted by $T(F, G_0, n)$. In~\cite{Noel23} it was proved that 
\[\max_{G_0\subset K_n^r}T(K_{m}^r,G_0,n)=\Theta(n^r)\] 
for all fixed $m\ge r+2\ge 5$.
In this paper, we are interested in the 
\emph{maximum running time} of the $F_r$-bootstrap percolation $T_r(n)$ defined as
\[
T_r(n) = \max_{G_0\subset K^r_n} T(F_r, G_0, n).
\]
In this case, \cite{Noel23} only established that for all $r\ge 3$, $T_r(n)\ge c_r n^{r-1}$ for some $c_r>0$. Our work closes the gap between this lower bound and the trivial upper bound $T_r(n)\le \binom{n}{r}$. More precisely, we answer Question~5.2 from~\cite{Noel23} in a strong sense, and disprove Conjecture~5.1 from the same paper suggesting that $T_3(n) = O(n^2)$. 

\begin{question}[Question 5.2 from~\cite{Noel23}]
Does there exist an integer $r_0$ such that, if $r\ge r_0$, then $T_r(n) = \Theta(n^r)$?
\end{question}

Our main result is the following lower bound.

\begin{theorem}\label{thm main}
For every $r\ge 3$ and $n\ge 2r^2$, we have
\[\frac{n^r}{2^{r+3}r^r}\le T_r(n)\le \frac{n^re^r}{r^r}.\]
\end{theorem}
Let us note that, while our result shows that the constant prefactor is $r^{-r} \mathrm{e}^{O(r)}$ as $r\to \infty$, our construction has not been optimised for this purpose, but rather for the clarity of the exposition. We also observe that the condition $n\ge 2r^2$ is imposed mostly for technical reasons and is only used in~\eqref{eq:ineq}.

\subsection{Outline of the proof}
\Cref{thm main} is proved as follows. We start with a simple construction showing that $T_3(2(4k-3)+1)\ge 8k^2-12k+4$ for any $k\ge 1$ (see \cref{sec:base}), which reproves the corresponding result of \cite{Noel23}. The construction (see \cref{fig:bipartite}) consists of two vertex-disjoint paths, each containing $4k-3$ vertices, and an additional vertex $v_1^3$. The edges in the initial graph are: $v_1^3$ together with any 2-edge in some of the paths, $v_1^3$ together with the initial vertices of the two paths, and some $3$-edges that do not contain $v_1^3$. This is done so that, firstly, only one edge is infected per time step, secondly, all edges which become infected contain $v_1^3$ and one vertex in each path, and moreover, the process is reversible in the sense that, if we replace the edge joining the beginnings of the two paths by the last edge infected by the process, the same edges become infected but in reverse order.

Based on this construction, we prove \cref{thm main} for $r=3$ by gluing $2k-1$ instances of it, as shown in \cref{fig:3level} (see \cref{prop:induction}). Namely, using an auxiliary vertex $v_2^3$, we propagate the infection from the end of the process for the vertex $v_1^3$ to the end of the same process for a new vertex $v_3^3$ instead of $v_1^3$, which will then run the other way around. Repeating this procedure $2k-2$ times, we obtain a cubic running time. The key feature of this strategy is that no edges that do not contain $v_1^3$ are created in the first construction, so we do not alter the configuration induced by the first $2(4k-3)$ vertices. Moreover, the same gluing technique allows us to extend the construction from $r$ to $r+1$ by adding a vertex to all edges, adding all edges that do not contain this vertex, and then gluing copies of this construction that differ only in this special vertex while keeping the remaining ones unchanged.

\section{Proof of Theorem~\ref{thm main}}
To start with, we fix a positive integer $k = k(n)$ to be chosen appropriately later. Then, for all integers $i\ge 1$, we define the vertex sets
\begin{align*}
A_i &{}= \{v_1^i, v_2^i, \dots, v_{4k-3}^i\},& A^*_i &{}= \bigcup_{j=1}^{i} A_j.
\end{align*}
In particular, the sets $(A_i)_{i\ge 1}$ are disjoint while the sets $(A_i^*)_{i\ge 1}$ form an increasing sequence with respect to inclusion.
Next, we consider a particular case of $F_r$-bootstrap percolation where a sequence of edges can be infected both in the given and in a reverse order under a suitable initial condition. 

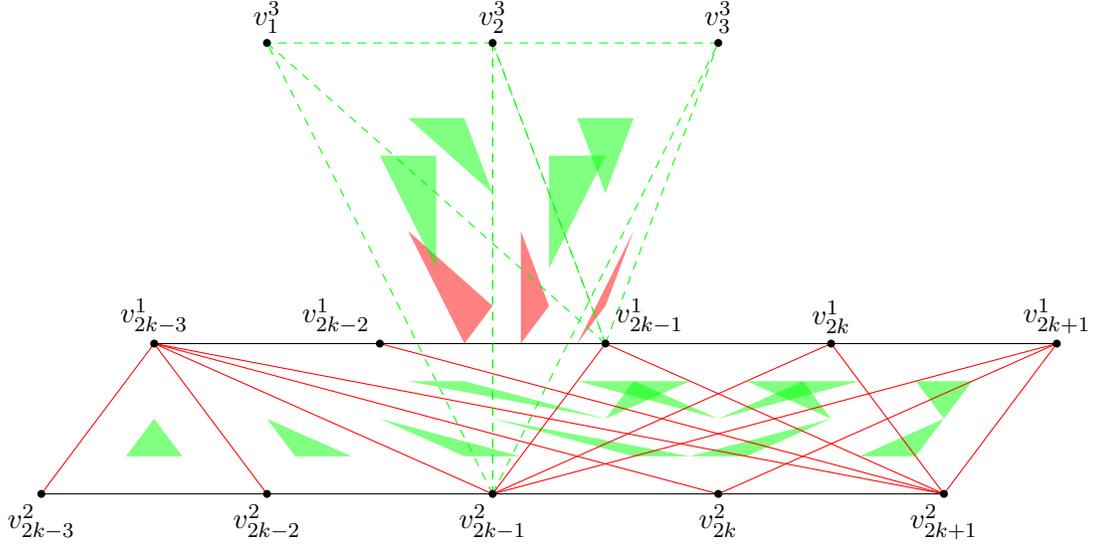
\begin{figure}
    \centering
    \begin{tikzpicture}[line cap=round,line join=round,>=triangle 45,x=1.5cm,y=1.0cm]
\fill[color=red,fill=red,fill opacity=0.5] (-0.75,3.5) -- (-0.25,2) -- (0,2.5) -- cycle;
\fill[color=red,fill=red,fill opacity=0.5] (0.25,2) -- (0.5,2.5) -- (0.25,3.5) -- cycle;
\fill[color=red,fill=red,fill opacity=0.5] (1,2.5) -- (0.75,2) -- (1.25,3.5) -- cycle;
\fill[color=green,fill=green,fill opacity=0.5] (0.75,5) -- (1,4) -- (1.25,5) -- cycle;
\fill[color=green,fill=green,fill opacity=0.5] (0.5,3) -- (0.5,4.5) -- (1,4.5) -- cycle;
\fill[color=green,fill=green,fill opacity=0.5] (-1,4.5) -- (-0.5,3) -- (-0.5,4.5) -- cycle;
\fill[color=green,fill=green,fill opacity=0.5] (-0.75,5) -- (-0.25,5) -- (0,4) -- cycle;
\fill[color=green,fill=green,fill opacity=0.5] (-3.25,0.5) -- (-2.75,0.5) -- (-3,1) -- cycle;
\fill[color=green,fill=green,fill opacity=0.5] (-2,1) -- (-1.75,0.5) -- (-1.25,0.5) -- cycle;
\fill[color=green,fill=green,fill opacity=0.5] (-0.25,0.5) -- (0.25,0.5) -- (-1,1) -- cycle;
\fill[color=green,fill=green,fill opacity=0.5] (1.25,0.5) -- (1.75,0.5) -- (0,1) -- cycle;
\fill[color=green,fill=green,fill opacity=0.5] (-0.75,1.5) -- (-0.25,1.5) -- (1,1) -- cycle;
\fill[color=green,fill=green,fill opacity=0.5] (0.75,1.5) -- (1.25,1.5) -- (2,1) -- cycle;
\fill[color=green,fill=green,fill opacity=0.5] (2.25,1.5) -- (3,1) -- (2.75,1.5) -- cycle;
\fill[color=green,fill=green,fill opacity=0.5] (4,1) -- (3.75,1.5) -- (4.25,1.5) -- cycle;
\fill[color=green,fill=green,fill opacity=0.5] (3.75,0.5) -- (4,1) -- (3.25,0.5) -- cycle;
\fill[color=green,fill=green,fill opacity=0.5] (3,1) -- (2.25,0.5) -- (1.75,0.5) -- cycle;
\fill[color=green,fill=green,fill opacity=0.5] (2,1) -- (3.25,1.5) -- (2.75,1.5) -- cycle;
\fill[color=green,fill=green,fill opacity=0.5] (1,1) -- (1.75,1.5) -- (1.25,1.5) -- cycle;
\draw[color=red] (0,0)-- (1,2);
\draw (-3,2)-- (5,2);
\draw (-4,0)-- (4,0);
\draw[color=red] (-4,0)-- (-3,2);
\draw[color=red] (5,2)-- (4,0);
\draw[color=red] (-3,2)-- (-2,0);
\draw[color=red] (-3,2)-- (0,0);
\draw[color=red] (-3,2)-- (2,0);
\draw[color=red] (-3,2)-- (4,0);
\draw[color=red] (4,0)-- (-1,2);
\draw[color=red] (4,0)-- (1,2);
\draw[color=red] (4,0)-- (3,2);
\draw[color=red] (2,0)-- (5,2);
\draw[color=red] (5,2)-- (0,0);
\draw[color=red] (0,0)-- (3,2);
\draw [dashed,color=green] (0,6)-- (0,0);
\draw [dashed,color=green] (0,0)-- (-2,6);
\draw [dashed,color=green] (-2,6)-- (1,2);
\draw [dashed,color=green] (1,2)-- (0,6);
\draw [dashed,color=green] (0,6)-- (-2,6);
\draw [dashed,color=green] (0,6)-- (1,2);
\draw [dashed,color=green] (1,2)-- (2,6);
\draw [dashed,color=green] (2,6)-- (0,6);
\draw [dashed,color=green] (0,6)-- (0,0);
\draw [dashed,color=green] (0,0)-- (2,6);
\fill (-3,2) circle (1.5pt) node[above]{$v_{2k-3}^1$};
\fill (-1,2) circle (1.5pt) node[above left]{$v_{2k-2}^1$};
\fill (1,2) circle (1.5pt) node[above right]{$v_{2k-1}^1$};
\fill (3,2) circle (1.5pt) node[above]{$v_{2k}^1$};
\fill (5,2) circle (1.5pt) node[above]{$v_{2k+1}^1$};
\fill (-4,0) circle (1.5pt) node[below]{$v_{2k-3}^2$};
\fill (-2,0) circle (1.5pt) node[below]{$v_{2k-2}^2$};
\fill (0,0) circle (1.5pt) node[below]{$v_{2k-1}^2$};
\fill (2,0) circle (1.5pt) node[below]{$v_{2k}^2$};
\fill (4,0) circle (1.5pt) node[below]{$v_{2k+1}^2$};
\fill (-2,6) circle (1.5pt) node[above]{$v_{1}^{3}$};
\fill (0,6) circle (1.5pt) node[above]{$v_{2}^{3}$};
\fill (2,6) circle (1.5pt) node[above]{$v_{3}^{3}$};
\end{tikzpicture}
    \caption{Illustration of the transition between the last step of stage $k-1$ for the vertex $v_1^3$ and the first step of the first stage for the vertex $v_3^3$. On this intermediate step, only the edge $v_{2k-1}^1 v_{2k-1}^2 v_2^3$ is infected. Note that 3-edges are represented by smaller homothetic copies of the corresponding triangles.}
    \label{fig:3level}
\end{figure}

\begin{definition}
\label{def:sequential}
Let $r\ge 3$ be an integer and $H$ be an $r$-graph with vertex set $V(H)\subset A_r^*$ and edge set $E(H)$. Given a sequence $(e_i)_{i=0}^T$ of edges in $E(H)$, we say that $H$ is \emph{$(e_i)_{i=0}^T$-sequential} if the $F_r$-bootstrap percolation process in the complete $r$-graph on $V(H)$ with initially infected $r$-graph:    
\begin{enumerate}[(i)]
    \item \label{property:i}$H$ runs for $T$ steps, infecting only the $r$-edge $e_i$ at step $i\in[1,T]$;
    \item \label{property:ii}$H\setminus\{e_0\}$ is stationary (does not infect any $r$-edge);
    \item \label{property:iii}$(H_1\cup \{e_{T}\})\setminus \{e_0\}$ infects only the $r$-edge $e_{T-i}$ at step $i\in[1,T]$.
\end{enumerate}
\end{definition}

\noindent
\cref{def:sequential} is central in the key construction presented in the next proposition.

\begin{proposition}\label{prop:induction}
Fix integers $r\ge 3$ and $k\ge 2$. Let $T_1\ge 2$ be an integer and $H_1$ be an $r$-graph with vertices $A^*_{r-1}\cup \{v^r_1\}$. Suppose that $H_1$ is $(e_i)_{i=0}^{T_1}$-sequential with $v_1^r\in\bigcap_{i=0}^{T_1}e_i$. For all $j\in [1,2k-1]$, let $H_{2j-1}$ be a copy of $H_1\setminus\{e_0\}$ obtained by replacing $v^r_1$ by $v_{2j-1}^r$. Moreover, define $e_i^- = e_i\setminus \{v_1^r\}$, and for all $j\in [1,k-1]$, let 
\begin{align*}
H_{4j-2}={}&\{e\subset \{v_{4j-3}^r, v^r_{4j-2}, v^r_{4j-1}\}\cup e_{T_1}^-\colon |e|=r, \{v_{4j-3}^r, v^r_{4j-1}\}\not\subset e, e_{T_1}^-\not\subset e\}\\
H_{4j}={}&\{e\subset \{v_{4j-1}^r, v^r_{4j}, v^r_{4j+1}\}\cup e_0^-\colon |e|=r, \{v_{4j-1}^r, v^r_{4j+1}\}\not\subset e, e_0^-\not\subset e\}
\end{align*}
(see \cref{fig:3level} for an illustration of $H_2$ with $r=3$, one choice of $H_1$ and $e_{T_1}^- = \{v_{2k-1}^1, v_{2k-1
}^2\}$, so that $H_2=\{v_{2k-1}^1v_1^3v_2^3,v_{2k-1}^2v_1^3v_2^3,v_{2k-1}^1v_2^3v_3^3,v_{2k-1}^2v_2^3v_3^3\}$). Let $H=\bigcup_{j=1}^{4k-3} H_j$. Then, for $T=(2k-1)T_1+4(k-1)$, there exists a sequence of $r$-edges $(e_i)_{i=0}^T$ extending $(e_i)_{i=0}^{T_1}$ such that $H$ is $(e_i)_{i=0}^T$-sequential. 
\end{proposition}

We prove \cref{thm main} by induction on $r$. The induction step is based on \cref{prop:induction}, leading to the following statement. 
\begin{proposition}
\label{prop:induction:hypothesis}
Let $k\ge 2$. There exist $r$-edges $(e_i)_{i=0}^{T_1}$ all containing $v_1^r$ and a $(e_i)_{i=0}^{T_1}$-squential $r$-graph $H_1$ with $V(H_1)=A_{r-1}^*\cup\{v_1^r\}$ and $T_1=(2k-1)^{r-3}(8k^2-12k+6)-2$.
\end{proposition}
The base case $r=3$ of \cref{prop:induction:hypothesis} is left to \cref{sec:base}.

\begin{proof}[Proof of \cref{prop:induction:hypothesis} assuming the case $r=3$]
We argue by induction on $r$, the base being established in \cref{sec:base}. Assume the induction hypothesis to be true for some $r\ge 3$. Then, \cref{prop:induction} provides us with an $(e_i)_{i=0}^T$-sequential $r$-graph denoted $H^r$ with $V(H^r)=A_r^*$ and 
\[T=(2k-1)T_1+4(k-1)=(2k-1)^{r-2}(8k^2-12k+6)-2.\]
Let $K$ be the complete $(r+1)$-graph with $V(K)=A_r^*$. Define the $(r+1)$-graph 
\[H^{r+1}_1=\{e\cup\{v^{r+1}_1\}\colon e\in H^r\}\cup K.\]
We claim that $H_1^{r+1}$ is $(e_i\cup\{v_1^{r+1}\})_{i=0}^T$-sequential. To see this, observe that $H_1^{r+1}\setminus\{e_0\cup\{v_1^{r+1}\}\}\supset K$ and proceed as follows. 

Fix any $(r+1)$-graph $H'\supset K$ with $V(H')=A_{r}^*\cup\{v_1^{r+1}\}$. Let $H'_i$ be the set of $(r+1)$-edges infected in the $F_{r+1}$-bootstrap percolation process on the complete graph on $V(H')$ and initially infected $(r+1)$-graph $H'$ at step $i$. Similarly, let $H''_i$ be the set of $r$-edges infected in the $F_r$-bootstrap percolation process on the complete $r$-graph with vertex set $A_r^*$ with initially infected $r$-graph $H''=\{e\setminus\{v_1^{r+1}\}\colon v_1^{r+1}\in e\in H'\}$ at step $i$. Then, by induction on $i$, it follows from the definition of these processes that
\[H''_i=\{e\setminus \{v_1^{r+1}\}\colon e\in H'_i\setminus K\}.\]
Applying this to $H'\in\{H_1^{r+1},H_1^{r+1}\setminus\{e_0\cup\{v_1^{r+1}\}\},(H_1^{r+1}\cup\{e_T\cup\{v_1^{r+1}\}\})\setminus\{e_0\cup\{v_1^{r+1}\}\}\}$, it follows from \cref{def:sequential} that $H^{r+1}_1$ is $(e_i\cup \{v^{r+1}_1\})_{i=0}^T$-sequential, as desired.
\end{proof}
We are now ready to conclude the proof of \cref{thm main}.
\begin{proof}[Proof of \cref{thm main}]
Fix $r\ge 3$ and set $k=\lfloor (n/r+3)/4\rfloor\ge n/(4r)-1/4$. For the lower bound, we apply \cref{prop:induction:hypothesis} and then \cref{prop:induction} to obtain an $r$-graph $H^r$ with $|A_r^*| = (r-1)(4k-3)+(4k-3) \le n$ vertices such that the $F_r$-bootstrap percolation process on the complete graph on $A_r^*$ with initially infected $r$-graph $H^r$ runs for
\begin{align}\nonumber
T&{}= (2k-1)^{r-2}(8k^2-12k+6)-2\ge (2k-1)^r\ge \left(\frac{n}{2r}-\frac{3}{2}\right)^r\ge \left(\frac{n}{2r}\right)^r\left(1-\frac{3r}{n}\right)^r\\
\label{eq:ineq}&{}\ge \left(\frac{n}{2r}\right)^r\left(1-\frac{3}{2r}\right)^r\ge \left(\frac{n}{2r}\right)^r\left(1-\frac{3}{2\cdot 3}\right)^3=\frac18\left(\frac{n}{2r}\right)^r.
\end{align}
steps, since $n\ge 2r^2$ and $r\ge 3$. Thus, the lower bound of \cref{thm main} is witnessed by the $r$-graph obtained by adding $n-|A_r^*|\ge 0$ isolated vertices to $H^r$. The upper bound simply states that the  process cannot run more than $\binom{n}{r}\le n^r/r!\le (ne/r)^r$ steps.
\end{proof}

\noindent
In the sequel, we call a vertex \emph{odd} (resp.\ \emph{even}) if its subscript is odd (resp.\ even).

\begin{proof}[Proof of \cref{prop:induction}]
First of all, let us show that the process in $H\setminus \{e_0\}$ is stationary. Suppose for contradiction that this is not the case, and suppose that some $(r+1)$-tuple of vertices $f$ spans exactly $r$ edges. To begin with, $f$ cannot contain two non-consecutive vertices in $A_r$ since they are not contained in a common $r$-edge of $H$. Moreover, $f$ cannot be entirely contained in $A_{r-1}^*$ since otherwise $H_1\setminus \{e_0\}$ would not be stationary. Two cases remain:
\begin{itemize}
    \item Suppose that $|f\cap A_r| = 1$. Then the vertex in $f\cap A_r$ must be even (otherwise, we obtain a contradiction with the stationarity of $H_1\setminus \{e_0\}$). However, even vertices in $A_r$ are not contained in edges with $r-1$ other vertices in $A_{r-1}^*$, so this case is impossible.
    \item Suppose that $f\cap A_r = \{v_{i}^r, v_{i+1}^r\}$ for some $i\in[1,4k-4]$. We further assume that $i=4j$ is divisible by 4, the other cases being treated identically. Since the only edges containing $v_{4j}^r$ are those in $H_{4j}$, necessarily $f=\{v_{4j}^r,v_{4j+1}^r\}\cup e_0^-$. However, in this case the two edges $\{v_{4j}^r\}\cup e_0^-$ and $\{v_{4j+1}^r\}\cup e_0^-$ are both missing, so we obtain the desired contradiction.
\end{itemize}

Now, for every $j\in [1,2k-1]$ and every $i\in [0, T_1]$, denote $e_{i,j} = e_i^-\cup \{v_{2j-1}^r\}$. In particular, $e_{i,1} = e_i$. Further set
\[H_{\infty} = H\cup \{e_{i,j}\colon i\in [0, T_1], j\in [1,2k-1]\}\cup \{\{v_{4j-2}^r\}\cup e_{T_1}^-\colon j\in [1,k-1]\}\cup \{\{v_{4j}^r\}\cup e_0^-\colon j\in [1,k-1]\}.\]

We claim that $H_\infty$ is also stationary. As above, the $r+1$ tuple $f$ spanning exactly $r$ edges of $H_\infty$ falls in one of the following two cases.
\begin{itemize}
    \item Suppose that $|f\cap A_r|=1$. Then, the vertex in $f\cap A_r$ must be even as otherwise we obtain a contradiction with the fact that $H_1\cup\{e_{i,1}\colon i\in[0,T_1]\}$ is stationary (which is true by the definition of $T_1$). However, even vertices in $A_r$ are contained in only one edge with $r-1$ other vertices in $A_{r-1}^*$, so this case is impossible.
    \item Suppose that $f\cap A_r = \{v_{i}^r, v_{i+1}^r\}$ for some $i\in[1,4k-4]$. Again, we assume that $i=4j$. Since the only edges containing $v^r_{4j}$ are those in $H_{4j}\cup\{\{v_{4j}^r\}\cup e_0^-\}$, necessarily $f=\{v_{4j}^r,v_{4j+1}^r\}\cup e_0^-$. However, $f$ already spans a copy of $F_r$ in $H_\infty$, which leads to a contradiction. 
\end{itemize}
Moreover, since $H\subset H_\infty$, all edges infected by $H$ are contained in $H_\infty$.

From \cref{def:sequential} one may conclude that for every $i\in [0,T_1]$, 
the only copies of $F_r$ in $H_1\cup \{e_j\}_{j=1}^{T_1}$ that $e_i$ participates in
are spanned by $e_{i-1}\cup e_i$ (if $i>0$) and by $e_i\cup e_{i+1}$ (if $i<T_1$). Moreover, note that:
\begin{itemize}
    \item for every $j\in [1,2k-1]$, $v_{2j-1}^r$ is contained in a common edge of $H_{\infty}$ with $v_{2j-2}^r$ (if $j\ge 2$) and with $v_{2j}^r$ (if $j\le 2k-2$), but with no other vertex in $A_r$,
    \item no edge among $(e_{i,j})_{i=1}^{T_1-1}$ participates in a copy of $F_r$ in $H_{\infty}$ with any of $v_{2j-2}^r$ and $v_{2j}^r$,
    \item for every $j\ge 2$, $e_{0,j}$ spans a copy of $F_r$ in $H_{\infty}$ together with $v_{2j-2}^r$ and $e_1\setminus e_0$ (if $j$ is odd) or $e_{T_1-1}\setminus e_{T_1}$ (if $j$ is even) only; for every $j\le 2k-2$, $e_{T_1, j}$ spans a copy of $F_r$ in $H_{\infty}$ together with $v_{2j}^r$ and $e_{T_1-1}\setminus e_{T_1}$ (if $j$ is odd) or $e_1\setminus e_0$ (if $j$ is even) only.
\end{itemize}
Furthermore, for every odd (resp.\ even) $j\in [1,2k-2]$, $v_{2j}^r\cup e_{T_1}^-$ (resp.\ $v_{2j}^r\cup e_0^-$) spans a copy of $F_r$ in $H_{\infty}$ together with $v_{2j-1}^r$ and $v_{2j+1}^r$ only. We conclude that in the sequence of edges
\begin{equation}
\label{eq:list}
(e_{i,1})_{i=0}^{T_1}, v_2^{r}\cup e_{T_1}^-, (e_{T_1-i,2})_{i=0}^{T_1}, v_4^r\cup e_0^-, (e_{i,3})_{i=0}^{T_1},\dots, (e_{T_1-i,2k-2})_{i=0}^{T_1},v_{4k-4}^r\cup e_0^-, (e_{i,2k-1})_{i=0}^{T_1},\end{equation}
every edge except the first one and the last one participate in exactly two copies of $F_r$ in $H_{\infty}$, one spanned together with the edge preceding it and another one with the edge succeeding it. Moreover, the first and the last edges participate in exactly one copy of $F_r$ in $H_{\infty}$, spanned together with $e_{1,1}$ and with $e_{T_1-1, 2k-1}$, respectively. Recalling \cref{def:sequential}, this completes the proof of \cref{prop:induction} since \eqref{eq:list} is an enumeration of all $T+1$ edges in $H_\infty\setminus (H\setminus\{e_0\})$.
\end{proof}

\section{The base case---proof of Proposition~\ref{prop:induction:hypothesis} for $r=3$}
\label{sec:base}

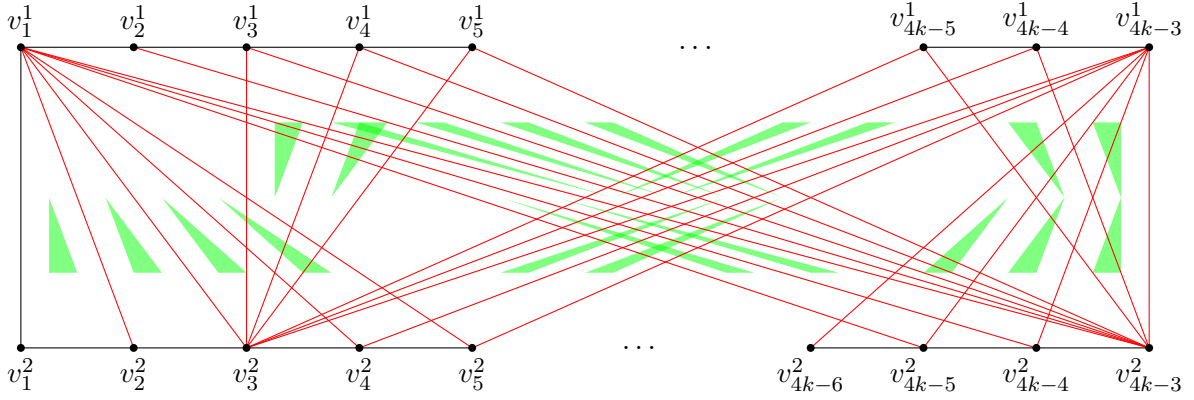
\begin{figure}
\begin{center}
\begin{tikzpicture}[line cap=round,line join=round,>=triangle 45,x=1.5cm,y=1cm]
\fill[color=green,fill=green,fill opacity=0.5] (0.5,1) -- (0.25,1) -- (0.25,2) -- cycle;
\fill[color=green,fill=green,fill opacity=0.5] (1.25,1) -- (0.75,2) -- (1,1) -- cycle;
\fill[color=green,fill=green,fill opacity=0.5] (1.25,2) -- (2,1) -- (1.75,1) -- cycle;
\fill[color=green,fill=green,fill opacity=0.5] (1.75,2) -- (2.75,1) -- (2.5,1) -- cycle;
\fill[color=green,fill=green,fill opacity=0.5] (4.25,2) -- (6.5,1) -- (6.25,1) -- cycle;
\fill[color=green,fill=green,fill opacity=0.5] (4.75,2) -- (7.25,1) -- (7,1) -- cycle;
\fill[color=green,fill=green,fill opacity=0.5] (2.75,3) -- (5.25,2) -- (3,3) -- cycle;
\fill[color=green,fill=green,fill opacity=0.5] (5.75,2) -- (3.75,3) -- (3.5,3) -- cycle;
\fill[color=green,fill=green,fill opacity=0.5] (6.25,2) -- (4.5,3) -- (4.25,3) -- cycle;
\fill[color=green,fill=green,fill opacity=0.5] (9.75,2) -- (9.5,3) -- (9.75,3) -- cycle;
\fill[color=green,fill=green,fill opacity=0.5] (9.25,2) -- (8.75,3) -- (9,3) -- cycle;
\fill[color=green,fill=green,fill opacity=0.5] (5,3) -- (6.75,2) -- (5.25,3) -- cycle;
\fill[color=green,fill=green,fill opacity=0.5] (9.75,2) -- (9.75,1) -- (9.5,1) -- cycle;
\fill[color=green,fill=green,fill opacity=0.5] (9.25,2) -- (8.75,1) -- (9,1) -- cycle;
\fill[color=green,fill=green,fill opacity=0.5] (6.25,2) -- (4.5,1) -- (4.25,1) -- cycle;
\fill[color=green,fill=green,fill opacity=0.5] (6.75,2) -- (5.25,1) -- (5,1) -- cycle;
\fill[color=green,fill=green,fill opacity=0.5] (8.75,2) -- (8,1) -- (8.25,1) -- cycle;
\fill[color=green,fill=green,fill opacity=0.5] (5.75,2) -- (7.5,3) -- (7.75,3) -- cycle;
\fill[color=green,fill=green,fill opacity=0.5] (5.25,2) -- (6.75,3) -- (7,3) -- cycle;
\fill[color=green,fill=green,fill opacity=0.5] (2.75,2) -- (3,3) -- (3.25,3) -- cycle;
\fill[color=green,fill=green,fill opacity=0.5] (2.25,2) -- (2.25,3) -- (2.5,3) -- cycle;
\draw (0,0)-- (0,4);
\draw[color=red] (0,4)-- (1,0);
\draw[color=red] (0,4)-- (2,0);
\draw[color=red] (0,4)-- (3,0);
\draw[color=red] (0,4)-- (4,0);
\draw[color=red] (0,4)-- (8,0);
\draw[color=red] (0,4)-- (9,0);
\draw[color=red] (0,4)-- (10,0);
\draw[color=red] (10,0)-- (1,4);
\draw[color=red] (10,0)-- (2,4);
\draw[color=red] (10,0)-- (3,4);
\draw[color=red] (4,4)-- (10,0);
\draw[color=red] (10,0)-- (8,4);
\draw[color=red] (9,4)-- (10,0);
\draw[color=red] (10,0)-- (10,4);
\draw[color=red] (10,4)-- (9,0);
\draw[color=red] (10,4)-- (8,0);
\draw[color=red] (10,4)-- (7,0);
\draw[color=red] (10,4)-- (4,0);
\draw[color=red] (10,4)-- (3,0);
\draw[color=red] (10,4)-- (2,0);
\draw[color=red] (2,0)-- (9,4);
\draw[color=red] (8,4)-- (2,0);
\draw[color=red] (2,0)-- (2,4);
\draw[color=red] (2,0)-- (3,4);
\draw[color=red] (2,0)-- (4,4);
\fill (0,0) circle (1.5pt) node[below]{$v_1^2$};
\fill (1,0) circle (1.5pt) node[below]{$v_{2}^2$};
\fill (2,0) circle (1.5pt) node[below]{$v_{3}^2$};
\fill (3,0) circle (1.5pt) node[below]{$v_{4}^2$};
\fill (4,0) circle (1.5pt) node[below]{$v_{5}^2$};
\fill (7,0) circle (1.5pt) node[below]{$v_{4k-6}^2$};
\fill (8,0) circle (1.5pt) node[below]{$v_{4k-5}^2$};
\fill (9,0) circle (1.5pt) node[below]{$v_{4k-4}^2$};
\fill (10,0) circle (1.5pt) node[below]{$v_{4k-3}^2$};
\fill (0,4) circle (1.5pt) node[above]{$v_1^1$};
\fill (1,4) circle (1.5pt) node[above]{$v_2^1$};
\fill (2,4) circle (1.5pt) node[above]{$v_3^1$};
\fill (3,4) circle (1.5pt) node[above]{$v_4^1$};
\fill (4,4) circle (1.5pt) node[above]{$v_5^1$};
\fill (8,4) circle (1.5pt) node[above]{$v_{4k-5}^1$};
\fill (9,4) circle (1.5pt) node[above]{$v_{4k-4}^1$};
\fill (10,4) circle (1.5pt) node[above]{$v_{4k-3}^1$};
\draw (0,0)--(4,0);
\draw (7,0)--(10,0);
\draw (5.5,0) node{\dots};
\draw (0,4)--(4,4);
\draw (6,4) node{\dots};
\draw (8,4)--(10,4);
\end{tikzpicture}
\end{center}
\caption{\label{fig:bipartite} The first stage of the bootstrap percolation process in $H_1$, where all 2-edges depicted are parts of 3-edges including $v_1^3$. The edges in $\cF$ (defined in \eqref{eq:def:F:odd}) and the ignition edge $e_0$ are represented by black 2-edges. They are infected initially, while red edges are created one by one by the process. Green triangles represent the 3-edges in $\cE_1$ defined in \eqref{eq E-s}); they are all infected initially as well. As in Fig.~\ref{fig:3level}, we draw a 3-edge as a smaller homothetic copy of the corresponding triangle.}
\end{figure}

In order to prove \cref{prop:induction:hypothesis} for $r=3$, our goal is to construct a sequential 3-graph $H_1$ with vertices $A_2^*\cup \{v_1^3\}$ and running time $8k^2-12k+4$. For all $i\in [1,k-1]$, define the disjoint edge sets
\begin{equation}
\begin{split}\label{eq E-s}
&\cE_{1,i} = \{v_{2i-1}^1 v_j^2 v_{j+1}^2: j\in [2i-1, 4k-2-2i]\},\\
&\cE_{2,i} = \{v_j^1 v_{j+1}^1 v_{4k-1-2i}^2: j\in [2i-1, 4k-2-2i]\},\\
&\cE_{3,i} = \{v_{4k-1-2i}^1 v_j^2 v_{j+1}^2: j\in [2i+1, 4k-2-2i]\},\\
&\cE_{4,i} = \{v_j^1 v_{j+1}^1 v_{2i+1}^2: j\in [2i+1, 4k-2-2i]\},\\
&\cE_i = \cE_{1,i}\cup \cE_{2,i}\cup \cE_{3,i}\cup \cE_{4,i}
\end{split}
\end{equation}
(see \cref{fig:bipartite} for an illustration of $\cE_1$). Also, denote $\cE = \bigcup_{i=1}^{k-1} \cE_i$. This is the set of edges with vertices in $A_2^*$ that we infect at the beginning.

In addition to $\cE$, some edges containing $v_1^3$ are also infected initially (see \cref{fig:bipartite}). Denote
\begin{equation}
\label{eq:def:F:odd}
\cF = \bigcup_{j=1}^{4k-4}\bigcup_{l=1}^2\{v_j^l v_{j+1}^l v_1^3\}.
\end{equation}
In other words, the set $\cF$ consists of all triplets containing the vertex $v_1^3$ and a consecutive pair of vertices in one of the paths $v^1_1 v^1_2 \dots v^1_{4k-3}$ and $v^2_1 v^2_2 \dots v^2_{4k-3}$. Finally, we need the \emph{ignition edge} $e_0 = v_1^1 v^2_1 v^3_1$ to start the growth process.

\begin{lemma}
\label{lem:base}
The set $\cE\cup\cF$ contains no $3$ edges with a total of $4$ vertices.
\end{lemma}
\begin{proof}
Let us first consider $4$-tuples of vertices in $A_2^*$. Note that all edges in $\cE$ contain two odd vertices and one even vertex. Hence, a set of 4 odd vertices in $A_2^*$ spans no edges, and a set of $4$ vertices in $A_2^*$ with at most $2$ of them odd spans at most $2$ edges. Finally, a set of $4$ vertices in $A_2^*$ with $3$ of them odd cannot span 3 edges since at least two of the odd vertices are either both in $A_1$, or both in $A_2$, and hence cannot participate in the same edge as they are not consecutive.

Now, consider $4$-tuples containing the vertex $v_1^3$. On the one hand, $A_1$ and $A_2$ are independent sets in the 3-graph $\cE$, and moreover the graph with edges $\{e\setminus v_1^3\colon v_1^3\in e\in\cF\}$ contains no triangles. Thus, every $4$-tuple containing $v_1^3$ and three vertices in $A_1$ (respectively in $A_2$) spans no edges of $\cE$ and at most 2 edges of $\cF$. It remains to observe that every 4-tuple containing $v_1^3$ and at least one vertex from both $A_1$ and $A_2$ cannot span more than 1 edge of $\cF$ (so 2 edges in total) since $\cF$ contains no edges intersecting both $A_1$ and $A_2$.
\end{proof}

We are now ready to conclude the proof of \cref{prop:induction:hypothesis} for $r=3$ by showing the following statement, illustrated in \cref{fig:bipartite} for $s=1$ and in \cref{fig:3level} for $s=k-1$. The proof is a straightforward but tedious verification.
\begin{lemma}\label{lem:H_1}
The $3$-graph $H_1=\cE\cup \cF\cup \{e_0\}$ is $(e_i)_{i=0}^{T_1}$-sequential for
\[\sum_{j=1}^{s-1} (16(k-j) - 4) < i\le \sum_{j=1}^s (16(k-j) - 4)\]
(which clearly exists since $T_1 = \sum_{j=1}^{k-1} (16(k-j) - 4)$).
Moreover, define 
\[\alpha = \alpha(i) = i - \sum_{j=1}^{s-1} (16(k-j) - 4)\in [1, 16(k-s)-4].\]
Then,
\begin{itemize}
    \item if $\alpha\in [1,4(k-s)]$, then $e_i = v_{2s-1}^1 v_{2s-1+\alpha}^2 v_1^3$;
    \item if $\alpha\in [4(k-s)+1, 8(k-s)]$, then $e_i = v_{6s-4k-1+\alpha}^1 v_{4k-2s-1}^2 v_1^3$;
    \item if $\alpha\in [8(k-s)+1, 12(k-s)-2]$, then $e_i = v_{4k-2s-1}^1 v_{12k-10s-1 - \alpha}^2 v_1^3$;
    \item if $\alpha\in [12(k-s)-1, 16(k-s)-4]$, then $e_i = v_{16k-14s-3-\alpha}^1 v_{2s+1}^2 v_1^3$.
\end{itemize}
\end{lemma}

Note that the sequence of $3$-edges indicated in Lemma~\ref{lem:H_1} may be alternatively described in the following more intuitive way (see Fig.~\ref{fig:bipartite}). At the beginning, the process first infects the edges $v_1^1v_j^2v_1^3$ for $j \in [2,4k-3]$, then $v_j^1v_{4k-3}^2v_1^3$ for $j\in [2,4k-3]$, then $v_{4k-3}^1v_j^2v_1^3$ for $j\in [3,4k-4]$ in reverse order, then $v_j^1v_3^2v_1^3$ for $j \in [3,4k-4]$ in reverse order.
At this point, $s$ increases from 1 to 2.
The process continues to propagate in a similar way except that, every time $s$ increases, the ``ends'' shift two vertices towards the middle (so $v_3^1$ plays the role of $v_1^1$ when $s$ is incremented for the first time).

\begin{proof}
To begin with, the $F_3$-bootstrap percolation process on the complete $3$-graph with vertices $A_2^*\cup\{v_1^3\}$ and initially infected $3$-graph $H_1\setminus \{e_0\} = \cE\cup \cF$ is stationary by \cref{lem:base}. Now, set $G_{-1} = \cE\cup \cF$, $G_0 = H_1$ and let the sequence $(G_i)_{i\ge 1}$ be defined as in~\eqref{eq:G_i}. We will show by induction on $i\ge 0$ that $G_i\setminus G_{i-1} = \{e_i\}$. 

For $i = 0$ the statement is clear by construction. Suppose that for some $i\ge 0$, the induction hypothesis is satisfied for $i$. Then, all edges in $G_{i+1}\setminus G_i$ must span 4 vertices with the edge $e_i$, and in particular must intersect $e_i$ in 2 vertices. We consider 4 cases according to the interval containing $\alpha =\alpha(i)$. For convenience, we extend the notation by setting $s(0)=1$ and $\alpha(0)=0$ (so that $\alpha(i)=0$ only when $s(i)=1$), which is consistent with $e_0=v^1_{2s-1}v^2_{2s-1+\alpha}v_1^3$).

If $\alpha\in [0,4(k-s)]$, then $e_i = v_{2s-1}^1 v_{2s-1+\alpha}^2 v_1^3$. One may easily check that the vertices in $A_2^*$ that complete the pair $v_{2s-1}^1 v_{2s-1+\alpha}^2$ to an edge in $G_{i-1}$ are 
\[\begin{cases}
v_{2s+\alpha}^2&\alpha = 0,\\
v_{2s-2+\alpha}^2,v_{2s+\alpha}^2&\alpha \in [1,4(k-s)-1],\\
v_{2s-2+\alpha}^2,v_{2s}^1&\alpha = 4(k-s).
\end{cases}\]
The vertices that complete the pair $v_{2s-1}^1 v_1^3$ to an edge in $G_{i-1}$ are 
\[\begin{cases}(v_{2s-1+j}^2)_{j\in [0, \alpha-1]}, v_{2s}^1&s=1,\\(v_{2s-1+j}^2)_{j\in [0, \alpha-1]}, v_{2s}^1,v^1_{2s-2}, (v_{2j-1}^2)_{j\in [2,s-1]\cup [2k-s+1, 2k-1]}
&s\ge 2.
\end{cases}\]
Finally, the vertices that complete the pair $v_{2s-1+\alpha}^2 v_1^3$ to an edge in $G_{i-1}$ are
\[\begin{cases}
v_{2s+\alpha}^2&\alpha=0,\\
(v_{2j-1}^1)_{j\in [1,s-1]\cup [2k-s+1,2k-1]},v^2_{2s-2+\alpha},v_{2s+\alpha}^2&\alpha\in[1,4k-5],\\
v^2_{2s-2+\alpha}&\alpha=4k-4.
\end{cases}\]

Then, the only vertices that complete at least 
two pairs are 
\[\begin{cases}v^2_{2}&\alpha=0,\\
v^2_{2s+\alpha},v^2_{2s+\alpha}&\alpha\in[1,4(k-s)-1],\\
v^2_{2s+\alpha},v^1_{2s}&\alpha=4(k-s).\end{cases}\]
However, for $\alpha\neq 0$, $\{v^2_{2s-2+\alpha}\}\cup e_i$ already spans a copy of $F_3$ in $G_i$, so the only edge that is generated from $G_i$ is $e_{i+1}$.

If $\alpha\in [4(k-s)+1, 8(k-s)]$, then $e_i = v_{6s-4k-1+\alpha}^1 v_{4k-2s-1}^2 v_1^3$. One may easily check that the vertices in $A_2^*$ that complete the pair $v_{6s-4k-1+\alpha}^1 v_{4k-2s-1}^2$ to an edge in $G_{i-1}$ are 
\[\begin{cases}v_{6s-4k-2+\alpha}^1,v_{6s-4k+\alpha}^1&\alpha\neq 8(k-s),\\
v_{4k-2s-2}^1,v_{4k-2s-2}^2&\alpha=8(k-s).\end{cases}\]
The vertices that complete the pair $v_{6s-4k-1+\alpha}^1 v_1^3$ to an edge in $G_{i-1}$ are
\[\begin{cases}v^1_{4k-4}&\alpha=8k-8,\\
(v_{2j-1}^2)_{j\in [2,s-1]\cup [2k-s+1,2k-1]}, v_{6s-4k-2+\alpha}^1,v_{6s-4k+\alpha}^1&\text{otherwise.}\end{cases}\]

Finally, the vertices that complete the pair $v_{4k-2s-1}^2 v_1^3$ to an edge in $G_{i-1}$ are 
\[\begin{cases}
(v_{5-4k+j}^1)_{j\in [4(k-1), \alpha-1]}, v_{4k-4}^2,&s=1,
\\(v_{6s-4k-1+j}^1)_{j\in [4(k-s), \alpha-1]}, (v_{2j-1}^{1})_{j\in [1,s-1
]\cup [2k-s+1, 2k-1]}, v_{4k-2s-2}^2,v_{4k-2s}^2&s\neq 1.
\end{cases}\]
Then, the only vertices that complete at least two pairs are
\[\begin{cases}
v^1_{4k-2s-2},v^2_{4k-2s-2}&\alpha=8(k-s),\\
v^1_{6s-4k-2+\alpha},v^1_{6s-4k+\alpha}&\alpha\neq 8(k-s).
\end{cases}\]
However, $\{v_{6s-4k-2+\alpha}^1\} \cup e_i$ already spans a copy of $F_3$ in $G_i$, so the only edge that is generated from $G_i$ is $e_{i+1}$.

If $\alpha\in [8(k-s)+1, 12(k-s)-2]$, then $e_i = v_{4k-2s-1}^1 v_{12k-10s-1-\alpha}^2 v_1^3$. One may easily check that the vertices in $A_2^*$ that complete the pair $v_{4k-2s-1}^1 v_{12k-10s-1-\alpha}^2$ to an edge in $G_{i-1}$ are 
\[\begin{cases}
v_{4k-2s-2}^1,v_{2s+2}^2&\alpha = 12(k-s)-2,\\
v_{12k-10s-2-\alpha}^2,v_{12k-10s-\alpha}^2&\alpha\neq 12(k-s)-2.
\end{cases}\]
The vertices that complete the pair $v_{4k-2s-1}^1 v_1^3$ to an edge in $G_{i-1}$ are 
\[\begin{cases}
(v_{12k-11-j}^{2})_{j\in [8k-8, \alpha-1]}, v_{4k-3}^2, v_{4k-4}^1,&s=1,\\
(v_{12k-10s-1-j}^{2})_{j\in [8(k-s), \alpha-1]}, (v_{2j-1}^2)_{j\in [2,s]\cup [2k-s,2k-1]}, v_{4k-2s-2}^1, v_{4k-2s}^1&s\neq 1.\end{cases}\]
The vertices that complete the pair $v_{12k-10s-1-\alpha}^2 v_1^3$ to an edge in $G_{i-1}$ are 
\[(v_{2j-1}^1)_{j\in [1,s]\cup [2k-s+1, 2k-1]}, v_{12k-10s-2-\alpha}^2, v_{12k-10s-\alpha}^2.\]
Then, the only vertices that complete at least two pairs are \[\begin{cases}
v_{4k-2s-2}^1,v_{2s+2}^2&\alpha=12(k-s)-2\\
v_{12k-10s-2-\alpha}^2,v_{12k-10s-\alpha}^2&\alpha\neq 12(k-s)-2.
\end{cases}\]
However, $\{v_{12k-10s-\alpha}^2\} \cup e_i$ already spans a copy of $F_3$ in $G_i$, so the only edge that is generated from $G_i$ is $e_{i+1}$.

If $\alpha\in [12(k-s)-1, 16(k-s)-4]$, then $e_i = v_{16k-14s-3-\alpha}^1 v_{2s+1}^2 v_1^3$. One may easily check that the vertices in $A_2^*$ that complete the pair $v_{16k-14s-3-\alpha}^1 v_{2s+1}^2$ to an edge in $G_{i-1}$ are
\[\begin{cases}
v^1_{2k}&\alpha=16(k-s)-4,s=k-1,\\
v_{2s+2}^1,v_{2s+2}^2&\alpha = 16(k-s)-4,s\neq k-1\\
v_{16k-14s-4-\alpha}^1,v_{16k-14s-2-\alpha}^1&\alpha\neq16(k-s)-4.
\end{cases}\]
The vertices that complete the pair $v_{16k-14s-3-\alpha}^1 v_1^3$ to an edge in $G_{i-1}$ are
\[(v_{2j-1}^2)_{j\in [2,s]\cup [2k-s, 2k-1]}, v_{16k-14s-4-\alpha}^1,v_{16k-14s-2-\alpha}^1.\]
Finally, the vertices that complete the pair $v_{2s+1}^2 v_1^3$ to an edge in $G_{i-1}$ are 
\[(v_{16k-14s-3-j}^1)_{j\in [12(k-s)-2, \alpha-1]}, (v_{2j-1}^1)_{j\in [1,s]\cup [2k-s+1,2k-1]}, v_{2s}^2, v_{2s+2}^2.\]
Then, the only vertices that complete at least two pairs are 
\[\begin{cases}
v^1_{2k}&\alpha=16(k-s)-4,s=k-1,\\
v_{2s+2}^1,v_{2s+2}^2&\alpha=16(k-s)-4,s\neq k-1,\\
v_{16k-14s-4-\alpha}^1,v_{16k-14s-2-\alpha}^1&\alpha\neq 16(k-s)-4.
\end{cases}\]
However, $\{v_{16k-14s-2-\alpha}^1\} \cup e_i$ already spans a copy of $F_3$ in $G_i$, so the only edge that is generated from $G_i$ is $e_{i+1}$, and the process stops when it infects $e_{T_1}$.

It remains to show that the $F_3$-bootstrap percolation process on the complete $3$-graph with vertices $A_2^*\cup\{v_1^3\}$ and initially infected $3$-graph $\cE\cup\cF\cup\{e_{T_1}\}$ infects the edges $(e_i)_{i=0}^{T_1-1}$ in reverse order. Notice that in the process started from $G_0$, every copy of $F_3$ that becomes completely infected contains two edges in $\cE\cup\cF$ and the edges $e_i$ and $e_{i+1}$ for some $i\ge 0$. Therefore, starting from $\cE\cup\cF\cup\{e_{T_1}\}$, the edges that do become infected are exactly $(e_i)_{i=0}^{T_1}$. Moreover, all copies of $F_3$ in $G_{T_1}$ are spanned by some of the vertex sets $(e_{i-1}\cup e_i)_{i=1}^{T_1}$, which implies that the bootstrap percolation with initial condition $\cE\cup\cF\cup\{e_{T_1}\}$ infects the edges $(e_i)_{i=0}^{T_1-1}$ in the reverse order.
\end{proof}

\paragraph{Acknowledgements.}
The authors are grateful to the `Minu Balkanski' foundation, where the current project was initiated, and to the anonymous referees for useful comments and remarks. IH was supported by the Austria Science Fund (FWF): P35428-N.

\paragraph{Note.}
A few days after the current article appeared on arxiv, an independent proof of a version of \cref{thm main} was found by Espuny D\'iaz, Janzer, Kronenberg and Lada~\cite{EspunyDiaz24}.

\bibliographystyle{plain}
\bibliography{Bib}

\end{document}